\documentclass[11pt,longtable,letterpaper,oneside,reqno]{amsart}
\usepackage{amsmath,amsfonts,amsthm,amssymb,amscd,mathtools,stmaryrd,url}
\usepackage{bbm}
\usepackage{epsfig}
\usepackage{graphicx}
\usepackage{latexsym}
\usepackage{longtable}
\usepackage{float}
\usepackage{hyperref}
\usepackage{color}

\DeclareFontFamily{OT1}{rsfs}{}
\DeclareFontShape{OT1}{rsfs}{n}{it}{<-> rsfs10}{}
\DeclareMathAlphabet{\mathscr}{OT1}{rsfs}{n}{it}

\newcommand{\fd}{f\otimes\chi_{d}}
\newcommand{\fod}{f\otimes\chi_{8d}}
\newcommand{\PhidX}{\Phi\left(\frac{8d}{X}\right)}
\newcommand{\PhidXa}{\Phi\left(\frac{8a^2d}{X}\right)}

\renewcommand{\Re}{{\operatorname{Re}}}
\renewcommand{\Im}{{\operatorname{Im}}}
\renewcommand{\a}{\alpha}

\newtheorem{prop}{Proposition}[section]
\newtheorem{thm}[prop]{Theorem}

\newtheorem{lem}[prop]{Lemma}

\newtheorem {defn }{Definition}

\numberwithin{equation}{section}
\usepackage[margin=1in]{geometry}
\usepackage{lipsum}

\allowdisplaybreaks

\begin{document}
\title{The second moment of derivatives of quadratic twists of  modular $L$-functions}

\begin{abstract}
We prove an asymptotic formula for the second moment of the first derivative of quadratic twists of modular $L$-functions with three leading order main terms. It improves the previous result of Kumar \emph{et al.} with the first main term. The proof is based on the large sieve type inequality established by Li, with a key input that we convert the problem into computing an asymptotic formula for the completed twisted modular $L$-functions with large shifts.
\end{abstract}
\author[Y. Jiang]{Yujiao Jiang}
\address{School of Mathematics and Statistics, Shandong University, Weihai, China}
\email{yujiaoj@sdu.edu.cn}

\author[Q. Shen]{Quanli Shen}
\address{SDU-ANU Joint Science College, Shandong University, Weihai, China}
\email{qlshen@outlook.com}

\author[Z. Tang]{Ziyang Tang}
\address{School of Mathematics and Statistics, Shandong University, Weihai, China}
\email{ziyangtang@mail.sdu.edu.cn}

%
\date{\today}
\maketitle


%
\section{Introduction}
The study of the moments of quadratic twists of modular $L$-functions is partially motivated by the Birch--Swinnerton-Dyer conjecture. Kolyvagin \cite{Kolyvagin} proved that the Mordell--Weil group of a modular elliptic curve $E$ over $\mathbb{Q}$ is finite if $L(1/2,E)\neq 0$ and $L(s,E\otimes \chi)$ has a simple zero at $s=1/2$ for some real Dirichlet character $\chi$. The second condition was independently verified by Bump--Friedberg--Hoffstein \cite{Bump-F-H} and Murty--Murty \cite{ Murty-Murty} by establishing an asymptotic for the first moment of the first derivative of quadratic twists of modular $L$-functions. 

To state our results more precisely, we introduce some notation and recall some standard facts. Let $f$ be a cusp form of weight $\kappa$ for $\operatorname{SL}_2(\mathbb{Z})$ and suppose that $f$ is a Hecke eigenform. The Fourier expansion of $f$ is 
\[
f(z)=\sum_{n=1}^\infty\lambda_{f}(n)n^{(\kappa-1)/2}e(nz),
\] 
with $\lambda_{f}(1)=1$. Deligne's bound gives $|\lambda_f (n)|\leq \tau(n)$ for all $n\geq 1$, where $\tau(n)$ is the divisor function. For $d$ a fundamental discriminant, 
let $\chi_d (\cdot):=\left(\frac{d}{\cdot}\right)$ denote the Kronecker symbol. Then $f\otimes\chi_d$ is a primitive Hecke eigenform of level $|d|^2$, with the $L$-function given by
\[
L(s,f\otimes\chi_d)=\sum_{n=1}^\infty\frac{\lambda_f (n)\chi_d (n)}{n^s}=\prod_{p}\left(1-\frac{\lambda_f (p)\chi_d (p)}{p^s}+\frac{\chi_d (p)^2}{p^{2s}}\right)^{-1}
\] 
for $\operatorname{Re}(s)>1$. The completed $L$-function is 
\[
\Lambda(s,f\otimes\chi_d)=\left(\frac{|d|}{2\pi}\right)^{s}\Gamma\left(s+\frac{\kappa-1}{2}\right)L(s,f\otimes\chi_d),
\] 
and satisfies the functional equation 
\begin{equation*}
\Lambda(s,f\otimes\chi_d)=i^{\kappa}\epsilon(d)\Lambda(1-s, f\otimes\chi_d),
\end{equation*}
where $\epsilon(d)=\chi_d(-1)=1$ if $d$ is positive and $\epsilon(d)=-1$ if $d$ is negative. We denote the root number by $\omega(\fd):=i^{\kappa}\epsilon(d)$. For convenience, it is also common to consider $\chi_{8d}$ with $d$ square-free integers.

The second moment of $L(1/2,f \otimes \chi_{8d})$ was asymptotically established by Soundararajan--Young \cite{Sound-Young} assuming the generalized Riemann hypothesis (GRH) for the case of $\omega(f \otimes \chi_{8d})=1$. 
It was recently proved unconditionally by Li \cite{Li}.  
For $\omega(f \otimes \chi_{8d})=-1$,  it is natural to study $L'(1/2,f \otimes \chi_{8d})$ due to $L(1/2,f \otimes \chi_{8d})= 0$. Based on the work of Soundararajan--Young \cite{Sound-Young}, Petrow \cite{Petrow} gave an asymptotic formula with two leading order main terms for the second moment of  $L'(1/2,f \otimes \chi_{8d})$ under GRH. Recently,  Kumar--Mallesham--Sharma--Singh \cite{Kumar} proved the first main term without GRH by using the large sieve type inequality established by Li \cite{Li}. In this paper, we improve the result of Kumar \emph{et al.} \cite{Kumar} by showing 
\begin{thm}\label{main-thm}
Assume $\kappa \equiv 2 \, (\operatorname{mod} 4)$. Let $\Phi:(0,\infty) \rightarrow \mathbb{R}$ be a smooth,
compactly supported function.  Then 
\begin{align*}
    \sideset{}{^*}\sum_{(d,2)=1}   L'(1/2,\fod)^2   \PhidX=  c_3 X (\log X)^3 + c_2 X (\log X)^2 +c_1 X \log X +  O(X (\log \log X)^5),
\end{align*}
where $\sum^{*}$ denotes the sum over square-free integers. Here 
\begin{align*}
    c_3& = \frac{2\widetilde{\Phi}(1) }{3\pi^2}
     L(1,\operatorname{sym}^2 f )^3
     Z_1(0,0),
\end{align*}
The factor  $Z_1(0,0)$ is defined in \eqref{def-Z-1},  and $\widetilde{\Phi}$ is the Mellin transform of $\Phi$ defined in \eqref{equ:mellin}. The coefficients $c_i$, $i=1,2$ can also be calculated precisely.
\end{thm}
The proof of Theorem \ref{main-thm} may extend to modular newforms for any Hecke congruence groups.  
The proof relies on the large sieve type inequality of Li \cite{Li}. The main input is the following observation. The derivative of an $L$-function morally carries an additional logarithmic factor compared with the  $L$-function itself, which may result in a larger error when computing moments.  To address this issue, we consider the following relation derived by the approximate functional equation (see \eqref{eq:AFE of L'}),
 \begin{align*}
       \sideset{}{^*}\sum_{\substack{(d,2)=1\\d \asymp X}}   L'(1/2,\fod)^2  
       \asymp \int_{(e_1)}\int_{(e_2)}
        \sideset{}{^*}\sum_{\substack{(d,2)=1\\d \asymp X}}  d^{-1}
       \Lambda(1/2+s_1,\fod) \Lambda(1/2+s_2,\fod)
      \frac{ds_1}{s_1^2}\frac{ds_2}{s_2^2}.
    \end{align*}
The above integrals lie on the vertical lines $\Re (s_i) =e_i>0$, $i=1,2$. The integrand contains  the  moment of completed $L$-functions with shifts $s_i$ without any differentiation.
 We then give it an asymptotic formula with an error roughly $\ll X(\log\log X)^3$ (see \eqref{equ:final}). To control the contribution of the factors $1/s_i^2$, we set $e_i \asymp 1/\log \log X$. The asymptotic of the moment in the integrand, together with the above relation, then implies Theorem \ref{main-thm}. 
The shifted moments of $L$-functions have been successfully applied in various problems (see \cite{Chandee,Conrey-Farmer-Keating-Rubinstein-Snaith,  NSW02,Sound-Young,Young}). In our setting, the shifts have relatively large real parts ($\asymp 1/\log \log X$) instead of smaller real parts ($\asymp 1/ \log X$) typically used in the literature.  We remark that very recently, Zhou \cite{Zhou} proved an asymptotic for the moment of derivatives of distinct twisted modular $L$-functions with an error similar to that in Theorem \ref{main-thm} using a different method. It would be very interesting to improve the error in the result of Zhou \cite{Zhou}, or in Theorem \ref{main-thm} here, to $o(X)$, which may require new ideas.

We briefly outline the argument here. We show some standard lemmas in Section \ref{sec2}. Sections \ref{sec3}--\ref{sec5} and the first half of Section \ref{sec6} are devoted to establishing \eqref{equ:final}. The proof of Theorem \ref{main-thm}  is completed in the second half of Section \ref{sec6}. More specifically,  in Section \ref{sec3}, we reduce the length of Dirichlet polynomials for the completed twisted $L$-functions by using the large sieve type inequality. In Section \ref{sec4}, we split the sum into the diagonal terms and the off-diagonal terms by using Poisson, and extract the main term from the diagonal terms. We give the off-diagonal terms an upper bound in Section \ref{sec5}.    

Throughout the paper, we use the notation $L:=\log \log X$ for brevity and always assume $ 0<|\Re(\alpha)|, |\Re(\beta)| \leq 1/L$ and $\alpha \neq \pm\beta$.

\section{Lemmas} \label{sec2}
We first introduce the approximate functional equation, which is referred to \cite[Theorem 5.3]{I-K}.
\begin{lem}\label{lem:AFE}
Assume $|\Re(\alpha)|<1/2$. Let $d>0$ be square-free.
Set  \begin{equation}\label{def:omega-alpha}
    \omega_{\alpha}(\xi):=\frac{1}{2\pi i}\int_{(1)}g_{\alpha}(s)\xi^{-s}\frac{ds}{s},
\end{equation}
where \begin{equation*}
    g_{\alpha}(s):=(2\pi)^{-s}\Gamma\left(\alpha+s+\frac{\kappa}{2}\right).
\end{equation*}
Then 
\begin{equation*}
    \Lambda\left(1/2+\alpha,\fod\right)=I_{\alpha}-I_{-\alpha},
\end{equation*}
where \begin{equation}\label{def:I_alpha}
    I_{\alpha}:=\left(\frac{8d}{2\pi}\right)^{1/2+\alpha}\sum_{n=1}^{\infty}\frac{\lambda_f (n)\chi_{8d} (n)}{n^{1/2+\alpha}}\omega_{\alpha}\left(\frac{n}{8d}\right).
\end{equation}
Also,
 \begin{equation}\label{eq:AFE of L'}
        L'(1/2,\fod)=\Gamma\left(\frac{\kappa}{2}\right)^{-1}\left(\frac{8d}{2\pi}\right)^{-1/2}\frac{2}{2 \pi i}\int_{(1)}\Lambda(1/2+s,\fod) e^{s^2}\frac{ds}{s^2}.
    \end{equation}
\end{lem}

\vspace{1em}

We now quote the Poisson summation formula in   \cite[Lemma 2.3]{Li}.
\begin{lem}\label{lem:Poisson}
    Let $F:(0,\infty) \rightarrow \mathbb{R}$ be a smooth,
compactly supported function. Let $n$ be an odd integer. Then \begin{equation}\label{eq_Poisson}
        \sum_{(d,2)=1}\left(\frac{d}{n}\right)F\left(\frac{d}{Z}\right)=\frac{Z}{2n}\left(\frac{2}{n}\right)\sum_{k\in\mathbb{Z}}(-1)^{k}G_{k}(n)\check{F}\left(\frac{kZ}{2n}\right).
    \end{equation}
Here the Gauss-like sum $G_k (n)$ is defined as \[
    G_{k}(n)=\left(\frac{1-i}{2}+\left(\frac{-1}{n}\right)\frac{1+i}{2}\right)\sum_{a\, (\operatorname{mod} n)}\left(\frac{a}{n}\right)e\left(\frac{ak}{n}\right),
    \] and the Fourier-type transform of $F$ is defined to be
    \begin{align}
        \check{F}(y)=&\int_{-\infty}^{\infty}(\cos+\sin)(2\pi xy)F(x)dx\nonumber\\ =& \frac{1}{2\pi i}\int_{(1/2)}\widetilde{F}(1-u)\Gamma(u)(\cos+\operatorname{sgn}(y)\sin)\left(\frac{\pi u}{2}\right)(2\pi |y|)^{-u}du,
        \label{def-Fv}
    \end{align}
    where \begin{align}
    \widetilde{F}(u)=\int_{0}^{\infty}F(x)x^{u}\frac{dx}{x}
    \label{equ:mellin}
    \end{align} 
    is the Mellin transform of $F$.
    
\end{lem}

In the following, we introduce the large sieve type inequalities as  shown in  Lemmas 5.3 and 6.3 of \cite{Li}.
Let $G$ be a smooth real-valued function with compact support on $[3/4,2]$ satisfying 
\[
    \begin{aligned}
       G(x)&=1\; \text{for all}\; x\in [1,3/2], \\
       G(x)+G(x/2)&=1\; \text{for all}\; x\in [1,3].
    \end{aligned}
\]
 Functions of this type appear in standard constructions of partitions of unity. Moreover, one can check that \[
G(x)+G(x/2)+\cdots+G(x/2^{J})=1
\] for $x\in[1,3\cdot 2^{J-1}]$ and is supported on $[3/4,2^{J+1}]$. Throughout this paper, we fix such a function $G$ with the properties above.

\begin{lem} \label{lem-large-sieve}
    Let $M,N\geq 1$, $t \in \mathbb{R}$ and $q\in \mathbb{N}^*$. Let $G$ be a smooth real function compactly supported on $[3/4,2]$, defined as above. Then 
    \begin{align*}
        \sideset{}{^\flat}\sum_{ 0<d\leq M } \left| \sum_{n=1}^\infty \frac{\lambda_f(n)}{n^{1/2+it}}  \left( \frac{d}{n}\right)G\left( \frac{n}{N}\right)\right|^2  &\ll M (1+|t|)^3\log (2+|t|),\\
        \sum_{\substack{(d,2)=1 \\ 0<d\leq M}} \left| \sum_{(n,q)=1} \frac{\lambda_f(n)}{n^{1/2+it}}  \left( \frac{8d}{n}\right)G\left( \frac{n}{N}\right)\right|^2  &\ll \tau(q)^5 M(1+|t|)^3\log (2+|t|),
    \end{align*}
where $ \sideset{}{^\flat}\sum$ is the sum over fundamental discriminants.
\end{lem}

\bigskip

\section{Reduce the length of Dirichlet polynomials} \label{sec3}
By Lemma \ref{lem:AFE},
\begin{align}
  &\sideset{}{^*}\sum_{(d,2)=1}\left(\frac{8d}{2\pi}\right)^{-1}\Lambda(1/2+\alpha,\fod)\Lambda(1/2+\beta,\fod)\PhidX\nonumber\\
  &=\sideset{}{^*}\sum_{(d,2)=1}\left(\frac{8d}{2\pi}\right)^{-1}(I_{\alpha}-I_{-\alpha})(I_{\beta}-I_{-\beta})\PhidX.\label{eq_1}
\end{align}
In  Sections \ref{sec3}--\ref{sec6}, we will prove an asymptotic formula for the above  moment (see\eqref{equ:final}).
 By symmetry, it suffices to consider 
 \begin{equation*}
    \sideset{}{^*}\sum_{(d,2)=1}\left(\frac{8d}{2\pi}\right)^{-1}I_{\alpha}I_{\beta}\PhidX.
\end{equation*} 
Define 
\begin{align}
I_{\alpha,U}:=\left(\frac{8d}{2\pi}\right)^{1/2+\alpha}\sum_{n_1=1}^{\infty}\frac{\lambda_{f}(n_1)\chi_{8d}(n_1)}{n_1^{1/2+\alpha}}\omega_{\alpha}\left(\frac{n_1}{8dU}\right),
\label{eq:I-11}
\end{align}
and 
\begin{align}\label{def:R_alpha}
    R_{\alpha}:=I_{\alpha}-I_{\alpha,U}.
\end{align}
Here $U= (\log X)^{-A}$ for $A>0$ a parameter chosen later.
Clearly,

\begin{align}
      I_{\alpha}I_{\beta} =I_{\alpha}I_{\beta,U} + I_{\alpha,U}I_{\beta}-I_{\alpha,U}I_{\beta,U} + R_{\alpha}R_{\beta} .
      \label{equ:split}
\end{align}

\begin{lem} \label{equ:reduce}
For  $|\Re(\alpha)|,|\Re(\beta)|\leq 1/L $,
\begin{align*}
     \sideset{}{^*}\sum_{(d,2)=1} \left(\frac{8d}{2\pi}\right)^{-1} R_{\alpha}R_{\beta} \PhidX \ll X (\log \log X)^3.
\end{align*}
\end{lem}
\begin{proof}


By \eqref{def:R_alpha},
\[
R_{\alpha}=\frac{1}{2\pi i}\int_{(1)} \left({\frac{8d}{2\pi}}\right)^{1/2+\alpha}g_{\alpha}(s_1)(8d)^{s_1}\frac{1-U^{s_1}}{s_1}\sum_{n_1=1}^{\infty}\frac{\lambda_f (n_1)\chi_{8d} (n_1)}{n_1^{1/2+\alpha+s_1}}\, ds_1.
\]
Inserting the smooth dyadic sum gives  
\[
R_{\alpha}=\sideset{}{^\#}\sum_{N_1}\frac{1}{2\pi i}\int_{(1)} \left({\frac{8d}{2\pi}}\right)^{1/2+\alpha} g_{\alpha}(s_1)(8d)^{s_1}\frac{1-U^{s_1}}{s}\sum_{n_1=1}^{\infty}\frac{\lambda_f (n_1)\chi_{8d} (n_1)}{n_1^{1/2+\alpha+s_1}}G\left(\frac{n_1}{N_1}\right)\, ds_1,
\]
where $\sideset{}{^\#}\sum$ means the sum over $N_1=2^j, j \geq 0$.
Let 
\[
V(x):=G(2x) + G(x)+ G(x/2).
\]
Note that $V(x)=1$ when $x \in [3/4,2]$. We add $V(x) $ in the sum, and then the Mellin inversion implies 
\[
\begin{aligned}
    \sum_{n_1=1}^{\infty}\frac{\lambda_f (n_1)\chi_{8d} (n_1)}{n_1^{1/2+\alpha+s_1}}G\left(\frac{n_1}{N_1}\right)&=
    \frac{1}{2\pi i}\int_{(1)} \sum_{n_1=1}^{\infty}\frac{\lambda_f (n_1)\chi_{8d} (n_1)}{n_1^{1/2+\alpha+s_1}}V\left(\frac{n_1}{N_1}\right)\left(\frac{N_1}{n_1}\right)^{z_1}\widetilde{G}(z_1)\,dz_1\\
    &=\frac{1}{2\pi i }\int_{(0)} \sum_{n_1=1}^{\infty}\frac{\lambda_f (n_1)\chi_{8d} (n_1)}{n_1^{1/2+z_1}}V\Big(\frac{n_1}{N_1}\Big)N_1^{z_1-\alpha-s_1}\widetilde{G}(z_1-\alpha-s_1)\,dz_1.
\end{aligned}
\]

We decompose $R_{\alpha}$ into two parts according to the range of  $N_1$. Let $R^-_{\alpha},R^+_{\alpha}$  denote the contribution from the range $N_1\leq X$ and $N_1>X$, respectively. We can similarly manipulate $R_\beta$. Next, we evaluate the case of $R^-_{\alpha}R^+_{\beta}$, and other cases are similar. Move the line to $\Re(s_1)=-2/L$ for $R^-_{\alpha}$ while moving to $\Re(s_2)=2/L$ for $R^+_{\beta}$ without encountering  poles.  It follows that 
\begin{align*}
\left(\frac{8d}{2\pi}\right)^{-\frac{1}{2}} R^-_{\alpha}\ll & d^{\Re(\alpha)-\frac{2}{L}} \sideset{}{^\#}\sum_{N_1\leq X} N_1^{-\Re(\alpha)+\frac{2}{L}} \int_{(-\frac{2}{L})}\int_{(0)}\left|\sum_{n_1=1}^{\infty}\frac{\lambda_f (n_1)\chi_{8d} (n_1)}{n_1^{1/2+z_1}}V\left(\frac{n_1}{N_1}\right)\right| \\
&\times \frac{|dz_1|}{1+|z_1-\alpha-s_1|^{10}}\frac{|ds_1|}{(1+|\alpha+s_1|^{10})|s_1|},
\end{align*}
and
\begin{align*}
\left(\frac{8d}{2\pi}\right)^{-\frac{1}{2}} R^+_{\beta}\ll &  d^{\Re(\beta)+\frac{2}{L}}  \sideset{}{^\#}\sum_{N_2> X} N_2^{-\Re(\beta)-\frac{2}{L}} \int_{(\frac{2}{L})}\int_{(0)} \left|\sum_{n_2=1}^{\infty}\frac{\lambda_f (n_2)\chi_{8d} (n_2)}{n_2^{1/2+z_2}}V\left(\frac{n_2}{N_2}\right)\right| \\
&\times \frac{|dz_2|}{1+|z_2-\beta-s_2|^{10}}\frac{|ds_2|}{(1+|\beta+ s_2|^{10})|s_2| }.
\end{align*}
By the Cauchy--Schwarz inequality and Lemma \ref{lem-large-sieve},  
\begin{align*}
   &  \sideset{}{^*}\sum_{(d,2)=1} \left(\frac{8d}{2\pi}\right)^{-1} R^-_{\alpha} R^+_{\beta}\\ 
   &\ll X^{\Re(\alpha)-\frac{2}{L}}     
      \sideset{}{^\#}\sum_{N_1\leq X} N_1^{-\Re(\alpha)+\frac{2}{L}} \cdot X^{\Re(\beta)+\frac{2}{L}}  \sideset{}{^\#}\sum_{N_2> X} N_2^{-\Re(\beta)-\frac{2}{L}} \cdot X (\log \log X)^{\varepsilon} \\
     &\ll X (\log \log X)^{2+\varepsilon}.
\end{align*}
In the above, we have used the bounds
\[
 \sideset{}{^\#}\sum_{N_1\leq X} N_1^{-\Re(\alpha)+\frac{2}{L}} \ll X^{-\Re(\alpha)+\frac{2}{L}} \log \log X, \qquad    \sideset{}{^\#}\sum_{N_2> X} N_2^{-\Re(\beta)-\frac{2}{L}} \ll X^{-\Re(\beta)-\frac{2}{L}} \log \log X,
\]
and 
\[
\int_{(\pm \frac{2}{L})} \frac{1}{(1+|\alpha+ s|^{10})|s|}\, |ds| \ll \log \log \log X,
\]
which can be obtained by considering  $|t|\leq |t-u|$ and $|t|> |t-u|$. 
This concludes the proof.
\end{proof}

\section{Evaluate diagonal terms } \label{sec4}
In Sections \ref{sec4}--\ref{sec6}, we will compute the sum of the term $I_{\alpha}I_{\beta,U}$ shown in \eqref{equ:split}. Other terms can be computed similarly. 
By  \eqref{def:I_alpha} and \eqref{eq:I-11},
\begin{align}
S:=&\sideset{}{^*}\sum_{(d,2)=1}\left(\frac{8d}{2\pi}\right)^{-1}I_{\alpha}I_{\beta,U}\PhidX \nonumber\\
=&\sideset{}{^*}\sum_{(d,2)=1}\left(\frac{8d}{2\pi}\right)^{\alpha+\beta}\sum_{n_1 ,n_2}\frac{\lambda_{f}(n_1)\lambda_{f}(n_2)\chi_{8d}(n_1 n_2)}{n_{1}^{1/2+\alpha}n_{2}^{1/2+\beta}}\omega_{\alpha}\left(\frac{n_1}{8d}\right)\omega_{\beta}\left(\frac{n_2}{8dU}\right)\PhidX.
\label{equ:S}
\end{align} 
The M\"{o}bius inversion implies 
    \begin{align}
    S=&\sum_{(a,2)=1}\mu(a)\sum_{(d,2)=1}\left(\frac{8a^{2}d}{2\pi}\right)^{\alpha+\beta}\sum_{(n_1 n_2 ,a)=1}\frac{\lambda_{f}(n_1)\lambda_{f}(n_2)\chi_{8d}(n_1 n_2)}{n_{1}^{1/2+\alpha}n_{2}^{1/2+\beta}}\nonumber\\
    &\times\omega_{\alpha}\left(\frac{n_1}{8a^{2}d}\right)\omega_{\beta}\left(\frac{n_2}{8a^{2}dU}\right)\PhidXa \nonumber\\ :=&S_1(\alpha,\beta;Y)+S_2(\alpha,\beta;Y).
    \label{def:S(a<Y)+S(a>Y)}
    \end{align}
Here $S_1(\alpha,\beta;Y)$ is the sum over $a<Y$, and $S_2(\alpha,\beta;Y)$ is the sum over $a\geq Y$, where $Y= (\log X)^{B}$  for $B>0$  a parameter chosen later.

\begin{lem} \label{lemagY}
In the region $|\Re(\alpha)|,|\Re(\beta)|\leq 1/L $,  $S_2(\alpha,\beta;Y)$ is holomorphic and 
    \begin{align*}
        S_2(\alpha,\beta;Y) \ll X^{1+2|\Re(\alpha)|+2|\Re(\beta)|} Y^{-1}(\log X)^{35} .
    \end{align*}
\end{lem}
\begin{proof}
   By \eqref{def:omega-alpha} and introducing $G(x)$ and $V(x)$ as in the proof of Lemma \ref{equ:reduce},
    \begin{align}
        S_2(\alpha,\beta;Y)= & \sideset{}{^\#}\sum_{N_1}\sideset{}{^\#}\sum_{N_2}\sum_{\substack{(a,2)=1\\a\geq Y}}\mu(a)\sum_{(d,2)=1}\left(\frac{8a^{2}d}{2\pi}\right)^{\alpha+\beta}\frac{1}{(2\pi i)^4}\int_{(|\Re(\beta)|+\frac{1}{\log X})}\int_{(|\Re(\alpha)|+\frac{1}{\log X})}\int_{(0)}\int_{(0)}\nonumber\\
        &\times N_1^{z_1-\alpha-s_1}N_2^{z_2-\beta-s_2}\left(8a^2d\right)^{s_1+s_2} U^{s_2}g_\alpha(s_1)g_\beta(s_2) \widetilde{G}(z_1-\alpha-s_1)\widetilde{G}(z_2-\beta-s_2)\frac{1}{s_1s_2} \nonumber\\
        &\times\sum_{(n_1n_2,a)=1}\frac{\lambda_f(n_1) \lambda_f(n_2)\chi_{8d} (n_1n_2)}{n_1^{1/2+z_1}n_2^{1/2+z_2}}V\left(\frac{n_1}{N_1}\right)V\left(\frac{n_2}{N_2}\right)\PhidXa \, dz_1\, dz_2\, ds_1 \, ds_2.
        \label{equ:a>Y-1}
    \end{align}
Here we have moved the lines of the integrals to $\Re(z_1)=\Re(z_2) = 0$ and $\Re(s_1)=|\Re(\alpha)|+{1}/{\log X}, \Re(s_2) = |\Re(\beta)|+{1}/{\log X}$.
By  the fact that $d \ll X/a^2$, the Cauchy--Schwarz inequality and Lemma \ref{lem-large-sieve},
\begin{align*}
    &\sum_{(d,2)=1}\left(\frac{8a^{2}d}{2\pi}\right)^{\alpha+\beta+s_1+s_2}\sum_{(n_1n_2,a)=1}\frac{\lambda_f(n_1) \lambda_f(n_2)\chi_{8d} (n_1n_2)}{n_1^{1/2+z_1}n_2^{1/2+z_2}}V\left(\frac{n_1}{N_1}\right)V\left(\frac{n_2}{N_2}\right)\PhidXa \\
    &\ll X^{1+2|\Re(\alpha)|+2|\Re(\beta)|}  \frac{\tau(a)^5}{a^2} (1+|\Im(z_1)|)^2 (1+|\Im(z_2)|)^2.
\end{align*}
Substituting this in \eqref{equ:a>Y-1} implies
\begin{align*}
   S_2(\alpha,\beta;Y)  &\ll X^{1+2|\Re(\alpha)|+2|\Re(\beta)|} (\log X)^2 \sum_{a\geq Y} \frac{\tau(a)^5}{a^2}  \sideset{}{^\#}\sum_{N_1}\sideset{}{^\#}\sum_{N_2}N_1^{-1/\log X}N_2^{-1/\log X}\\
    &\ll X^{1+2|\Re(\alpha)|+2|\Re(\beta)|} Y^{-1}(\log X)^{35}.
\end{align*}
\end{proof}


Next, we evaluate $S_1(\alpha,\beta;Y)$  in \eqref{def:S(a<Y)+S(a>Y)}.
By \eqref{eq_Poisson},
    \begin{align}
       S_1(\alpha,\beta;Y)=&\frac{1}{2}\sum_{\substack{(a,2)=1\\ a<Y}}\mu(a)\left(\frac{8a^{2}}{2\pi}\right)^{\alpha+\beta}\sum_{(n_1 n_2 ,2a)=1}\frac{\lambda_{f}(n_1)\lambda_{f}(n_2)}{n_{1}^{1/2+\alpha}n_{2}^{1/2+\beta}}\sum_{k\in\mathbb{Z}}(-1)^{k}\frac{G_{k}(n_1 n_2)}{n_1 n_2}\check{E}\left(\frac{k}{2n_1 n_2}\right)\nonumber\\ :=&S(k=0)+S(k\neq 0),\label{def:S(k=0)+S(kne0)}
    \end{align}
where 
\begin{align}
E(x):=x^{\alpha+\beta}\omega_{\alpha}\left(\frac{n_1}{8a^{2}x}\right)\omega_{\beta}\left(\frac{n_2}{8a^{2}xU}\right)\Phi\left(\frac{8a^{2}x}{X}\right),
\label{equ:smoothtaking}
\end{align}
and $S(k=0)$ means the term with $k=0$, and $S(k\neq 0)$ are the terms with $k\neq 0$.
By Lemma \ref{lem:Poisson}, $G_{0}(n_1 n_2)=\phi(n_1 n_2)$ when $n_1 n_2$ is a square, and vanishes otherwise. Thus,
\begin{align}
    S(k=0)
    = &\frac{X^{1+\alpha+\beta}}{16}\sum_{\substack{(a,2)=1\\ a<Y}}\frac{\mu(a)}{a^2}({2\pi})^{-\alpha-\beta}\sum_{\substack{n_1 n_2 =\square\\ (n_1 n_2 ,2a)=1}}\frac{\lambda_{f}(n_1)\lambda_{f}(n_2)}{n_{1}^{1/2+\alpha}n_{2}^{1/2+\beta}}\prod_{p|n_1 n_2}\left(1-\frac{1}{p}\right)\nonumber\\
   &\times  \int_{0}^\infty  x^{\alpha+\beta} \omega_\alpha\left(\frac{n_1}{xX}\right) \omega_\beta\left(\frac{n_2}{xXU}\right)\Phi(x) \,dx.
    \label{equ:dia-001}
\end{align}
Write 
\begin{align}
    S_3(\alpha,\beta;Y)= &\frac{X^{1+\alpha+\beta}}{16}\sum_{\substack{(a,2)=1\\ a\geq Y}}\frac{\mu(a)}{a^2}({2\pi})^{-\alpha-\beta}\sum_{\substack{n_1 n_2 =\square\\ (n_1 n_2 ,2a)=1}}\frac{\lambda_{f}(n_1)\lambda_{f}(n_2)}{n_{1}^{1/2+\alpha}n_{2}^{1/2+\beta}}\prod_{p|n_1 n_2}\left(1-\frac{1}{p}\right)\nonumber\\
    &\times   \int_{0}^\infty  x^{\alpha+\beta} \omega_\alpha\left(\frac{n_1}{xX}\right) \omega_\beta\left(\frac{n_2}{xXU}\right)\Phi(x) \,dx.
    \label{def-s3}
\end{align}
Note that $S_3(\alpha,\beta;Y)$ is a holomorphic function for variables $\alpha,\beta$. Switching the order of the sum over $a$ and the sum over $n_1,n_2$, we see 
\begin{align}
     S_3(\alpha,\beta;Y) &\ll X^{1+\Re(\alpha)+\Re(\beta)}Y^{-1}\sum_{\substack{n_1 n_2 =\square\\ (n_1 n_2 ,2)=1}}\frac{\tau(n_1)\tau(n_2)}{n_{1}^{1/2+\Re(\alpha)}n_{2}^{1/2+\Re(\beta)}}  \int_{0}^\infty \left|\omega_\alpha\left(\frac{n_1}{xX}\right) \omega_\beta\left(\frac{n_2}{xXU}\right)\right|\Phi(x)\,dx\nonumber\\
    &\ll X^{1+\Re(\alpha)+\Re(\beta)}Y^{-1} \sum_{\substack{n_1 n_2 =\square\\ n_1, n_2 \ll X^2}}\frac{\tau(n_1)\tau(n_2)}{n_{1}^{1/2+\Re(\alpha)}n_{2}^{1/2+\Re(\beta)}}+ X^{-1}\nonumber\\
    &\ll X^{1+3|\Re(\alpha)|+3|\Re(\beta)|}Y^{-1} (\log X)^{10}.
    \label{equ:upperS3}
\end{align}
This combined with \eqref{equ:dia-001} and the identity
\begin{align*}
        \sum_{(a,2n_1 n_2)=1}\frac{\mu(a)}{a^2}=\frac{8}{\pi^2}\prod_{p|n_1 n_2}\left(1-\frac{1}{p^2}\right)^{-1},
\end{align*}
gives 
\begin{align}
    S(k=0)=&\frac{X^{1+\alpha+\beta}}{2\pi^2}({2\pi})^{-\alpha-\beta} \frac{1}{(2\pi i)^2}\int_{(1)}\int_{(1)} X^{s_1+s_2}U^{s_2} \sum_{\substack{n_1 n_2 =\square\\ (n_1 n_2 ,2)=1}}\frac{\lambda_{f}(n_1)\lambda_{f}(n_2)}{n_{1}^{1/2+\alpha+s_1}n_{2}^{1/2+\beta+s_2}}\prod_{p|n_1 n_2}\frac{p}{p+1}\nonumber\\
    &\times g_\alpha(s_1) g_\beta(s_2)
    \widetilde{\Phi}(\alpha+\beta+s_{1}+s_{2}+1)\frac{ds_1} {s_1}\frac{ds_2}{s_2}
    + 
    S_3(\alpha,\beta;Y).\label{eq_4-1}
\end{align}
We keep the original form \eqref{def-s3} of $S_3(\alpha,\beta; Y)$  above since the upper bound in \eqref{equ:upperS3} becomes too large when $\Re(\alpha),\Re(\beta) \asymp 1/L$. In Section \ref{sec6},   we will restrict $|\Re(\alpha)|,|\Re(\beta)| \ll  1/\log X$ to give a sufficiently small bound for $S_3(\alpha,\beta;Y)$. To proceed,
we need the following lemma (see \cite[Lemma 5.5]{Li}).
\begin{lem} \label{lem-dia}
For $\Re(z_1),\Re(z_2)>0$,
\begin{align}
&\sum_{\substack{n_1 n_2 =\square\\ (n_1 n_2 ,2)=1}}\frac{\lambda_{f}(n_1)\lambda_{f}(n_2)}{n_{1}^{1/2+z_1}n_{2}^{1/2+z_2}}\prod_{p|n_1 n_2}\frac{p}{p+1}\nonumber \\
&=\zeta(1+z_1+z_2) L(1+2z_1,\operatorname{sym}^2 f ) L(1+2z_2,\operatorname{sym}^2 f ) L(1+z_1+z_2,\operatorname{sym}^2 f )
Z_1(z_1,z_2), 
\label{def-Z-1}
\end{align}
where $Z_1(z_1,z_2)$ converges absolutely and is holomorphic for $\Re(z_1),\Re(z_2)>-1/4+\varepsilon$.
\end{lem}
By \eqref{eq_4-1} and Lemma \ref{lem-dia},
\begin{align}
    S(k=0)=&\frac{X^{1+\alpha+\beta}}{2\pi^2}({2\pi})^{-\alpha-\beta} \frac{1}{(2\pi i)^2}\int_{(\frac{1}{10})}\int_{(\frac{1}{10})} X^{s_1+s_2}U^{s_2} 
    \zeta(1+\alpha+\beta+s_1+s_2)\nonumber\\
    &\times L(1+2\alpha+2s_1,\operatorname{sym}^2 f ) L(1+2\beta+2s_2,\operatorname{sym}^2 f ) L(1+\alpha+\beta+s_1+s_2,\operatorname{sym}^2 f )
   \nonumber\\
    &\times  Z_1(\alpha+s_1,\beta+s_2) g_\alpha(s_1) g_\beta(s_2)
    \widetilde{\Phi}(\alpha+\beta+s_{1}+s_{2}+1)\frac{ds_1} {s_1}\frac{ds_2}{s_2}  + 
    S_3(\alpha,\beta;Y) . \label{equ:dia-2}
\end{align}
Move the line of the integral to $\Re(s_1)=-1/5$ with  poles  at $s_1= 0,-\alpha-\beta-s_2$. The integral on the new line is $\ll X^{10/11}$. The residue at $s_1=-\alpha-\beta-s_2$ for the integral in \eqref{equ:dia-2} is 
\begin{align*}
  &\frac{X^{1+\alpha+\beta}}{2\pi^2}({2\pi})^{-\alpha-\beta} \frac{1}{2\pi i}\int_{(\frac{3}{L})} X^{-\alpha-\beta}U^{s_2} 
    L(1-2\beta-2s_2,\operatorname{sym}^2 f ) L(1+2\beta+2s_2,\operatorname{sym}^2 f ) L(1,\operatorname{sym}^2 f )
   \nonumber\\
    &\quad\times  Z_1(-\beta-s_2,\beta+s_2) g_\alpha(-\alpha-\beta-s_2) g_\beta(s_2)
    \widetilde{\Phi}(1) \frac{1}{-\alpha-\beta-s_2}\frac{ds_2}{s_2}\nonumber \\
   & \ll X(\log \log X)^2 .
\end{align*}
Similarly, the  residue at $s_1=0$ for the integral in \eqref{equ:dia-2} is  
\begin{align*}
    &\frac{X^{1+\alpha+\beta}}{2\pi^2}({2\pi})^{-\alpha-\beta} \frac{1}{2\pi i}\int_{(\frac{1}{10})} X^{s_2}U^{s_2} 
    \zeta(1+\alpha+\beta+s_2)\times L(1+2\alpha,\operatorname{sym}^2 f ) L(1+2\beta+2s_2,\operatorname{sym}^2 f )\nonumber\\
    & \times L(1+\alpha+\beta+s_2,\operatorname{sym}^2 f )
     Z_1(\alpha,\beta+s_2) g_\alpha(0) g_\beta(s_2)
    \widetilde{\Phi}(\alpha+\beta+s_{2}+1)\frac{ds_2}{s_2}.
\end{align*}
By moving the line of the integral to $\Re(s_2)=-1/10$ with  poles at $s_2= 0,-\alpha-\beta$. The integral on the new line is  $\ll X^{10/11}$. The residue at $s_2= -\alpha-\beta$  is 
\begin{align*}
    &  \frac{X^{1+\alpha+\beta}}{2\pi^2}({2\pi})^{-\alpha-\beta}  X^{-\alpha-\beta}U^{-\alpha-\beta} 
   L(1+2\alpha,\operatorname{sym}^2 f ) L(1-2\alpha,\operatorname{sym}^2 f )\nonumber\\
    & \times L(1,\operatorname{sym}^2 f )
     Z_1(\alpha,-\alpha) g_\alpha(0) g_\beta(-\alpha-\beta)
    \widetilde{\Phi}(1)\frac{1}{-\alpha-\beta}\ll X {|\alpha+\beta|}^{-1}.
\end{align*}
The residue at $s_2= 0$  is 
\begin{align}
    M(\alpha,\beta):=&\frac{X^{1+\alpha+\beta}}{2\pi^2}({2\pi})^{-\alpha-\beta} 
    \zeta(1+\alpha+\beta) L(1+2\alpha,\operatorname{sym}^2 f ) L(1+2\beta,\operatorname{sym}^2 f )\nonumber\\
    & \times L(1+\alpha+\beta,\operatorname{sym}^2 f )
     Z_1(\alpha,\beta) g_\alpha(0) g_\beta(0)
    \widetilde{\Phi}(\alpha+\beta+1).  \label{equ:dia-5}
\end{align}
It follows from the discussion below \eqref{equ:dia-2} that 
\begin{lem} \label{lem:asyk=0}
We have 
\begin{align*}
        S(k=0)&= M(\alpha,\beta) + S_3(\alpha,\beta;Y)
        + O(X  {|\alpha+\beta|}^{-1}) + O(X (\log \log X)^2 ),
\end{align*}
where $M(\alpha,\beta), S_3(\alpha,\beta;Y) $ are  defined in \eqref{equ:dia-5}, \eqref{def-s3}, respectively. In addition, in the region $0<|\Re(\alpha)|,|\Re(\beta)| <1/L$, $S_3(\alpha,\beta;Y)$ is holomorphic and is bounded by 
\[
S_3(\alpha,\beta;Y) \ll X^{1+3|\Re(\alpha)|+3|\Re(\beta)|}Y^{-1} (\log X)^{10}.
\]
\end{lem}

\section{Evaluate off-diagonal terms} \label{sec5}
By the definition of $S(k\neq 0)$ in \eqref{def:S(k=0)+S(kne0)}, \eqref{equ:smoothtaking} and  \eqref{def-Fv}, it follows that 
\begin{align*}
     S(k\neq 0)=&\frac{X}{2}\sum_{\substack{(a,2)=1\\ a<Y}}\mu(a)
     \frac{1}{(2\pi i)^3} \int_{(1)}\int_{(1)}\int_{(1)} X^{\alpha+\beta+s_1+s_2-u}U^{s_2}(8a^2)^{-1+u}  \\
    &\times  (2\pi)^{-\alpha-\beta-u} 2^{u}\widetilde{\Phi}(1+\alpha+\beta +s_1 + s_2-u)g_{\alpha}(s_1)g_{\beta}(s_2)\Gamma(u)(\cos+\mathrm{sgn}(k)\sin)\left(\frac{\pi u}{2}\right)\\
    &\times 
    \sum_{ k\neq 0} \frac{(-1)^{k}}{|k|^u}\sum_{(n_1 n_2 ,2a)=1}\frac{\lambda_{f}(n_1)\lambda_{f}(n_2)}{n_{1}^{1/2+\alpha+s_1-u}n_{2}^{1/2+\beta+s_2-u}} \frac{G_{k}(n_1 n_2)}{n_1 n_2}\frac{1}{s_{1}s_{2}}\,ds_1\,ds_2\,du.
\end{align*}
Changing the variables $\alpha+s_1 - u \mapsto s_1$, $\beta+s_2 - u \mapsto s_2$, and letting $k=k_1k_2^2$ with $k_1\in \mathbb{Z}$ square-free, we have
\begin{align}
     S(k\neq 0)=&X\sum_{\substack{(a,2)=1\\ a<Y}}\mu(a)
     \frac{1}{(2\pi i)^3} \int_{(1)}\int_{(1)}\int_{(1)}   X^{s_1+s_2+u}U^{-\beta+s_2+u}a^{-2+2u} \sideset{}{^*}\sum_{ k_1\neq 0}\frac{1}{|k_1|^u} \sum_{k_2=1}^\infty\frac{(-1)^{k_1k_2}}{k_2^{2u}} \nonumber\\
    &\times 
    \sum_{(n_1 n_2 ,2a)=1}\frac{\lambda_{f}(n_1)\lambda_{f}(n_2)}{n_{1}^{1/2+s_1}n_{2}^{1/2+s_2}} \frac{G_{k_1k_2^2}(n_1 n_2)}{n_1 n_2} \mathcal{K}(s_1,s_2,u;k_1,\alpha,\beta)\,ds_1\,ds_2\,du,
    \label{equ:off-1}
\end{align}
where
\begin{align*}
    \mathcal{K}(s_1,s_2,u;k_1,\alpha,\beta):=& 2^{-1}8^{-1+u}  (2\pi)^{-\alpha-\beta-u} 2^{u}\widetilde{\Phi}(1+s_1+s_2+ u)\Gamma(u)g_{\alpha}(-\alpha+s_1+u)\\
    &\times g_{\beta}(-\beta+s_2+u) (\cos+\mathrm{sgn}(k_1)\sin)\left(\frac{\pi u}{2}\right)\frac{1}{(-\alpha+s_{1}+u )(-\beta+s_{2}+u)}.
\end{align*}
To proceed, we need the following lemma (see    \cite[Lemma 2.5]{Li}).
\begin{lem} \label{lem:off-euler}
    Let $k_1$ be square-free. Let $m=k_1$ if $k_1 \equiv 1\, (\operatorname{mod} 4)$ and $m=4k_1$ if $k_1 \equiv 2,3\, (\operatorname{mod} 4)$. Then for $\Re(z_i)>1/2$, $i=1,2,3$,
\begin{align*}
  & \sum_{k_2=1}^\infty\frac{1}{k_2^{2z_3}}\sum_{(n_1 n_2 ,2q)=1}\frac{\lambda_{f}(n_1)\lambda_{f}(n_2)}{n_{1}^{z_1}n_{2}^{z_2}} \frac{G_{k_1k_2^2}(n_1 n_2)}{n_1 n_2}
\\&= L(1/2+z_1, f\otimes \chi_m) L(1/2+z_2, f \otimes \chi_m) Y(z_1,z_2,z_3;k_1,q),
\end{align*}
where
\begin{align*}
    Y(z_1,z_2,z_3;k_1,q) = \frac{Z_2(z_1,z_2,z_3)}{\zeta(1+z_1+z_2) L(1+2z_1,\operatorname{sym}^2 f) L(1+z_1+z_2,\operatorname{sym}^2 f)
    L(1+2z_2,\operatorname{sym}^2 f)
    }
\end{align*}
and 
$Z_2(z_1,z_2,z_3) := Z_2(z_1,z_2,z_3;k_1,q)$ is holomorphic for $\Re(z_i)\geq -\delta/2$, $i=1,2$ and $\Re(z_3)\geq 1/2+ \delta $ for any $0<\delta<1/2$. In addition, $Z_2(z_1,z_2,z_3)\ll \tau(q)$  in the same region.  
\end{lem}
We next evaluate \eqref{equ:off-1} for the case  that $k_1$ are positive odd numbers, and the computation for the other cases is similar. Since   $G_k(n) = G_{4k}(n)$ for odd $n$, by the inclusion-exclusion, we deduce
\begin{align*}
    \sum_{k_2=1}^\infty\frac{(-1)^{k_2}}{k_2^{2u}}{G_{k_1k_2^2}(n_1 n_2)}
    =
    (2^{1-2u}-1) \sum_{k_2=1}^\infty\frac{1}{k_2^{2u}}{G_{k_1k_2^2}(n_1 n_2)}.
\end{align*}
This combined with \eqref{equ:off-1} and Lemma \ref{lem:off-euler} gives 
\begin{align*}
     S(&k\neq 0, k_1 \operatorname{odd})\nonumber\\
     =&X\sum_{\substack{(a,2)=1\\ a<Y}}\mu(a)
     \frac{1}{(2\pi i)^3} \int_{(1)}\int_{(1)}\int_{(1)}   X^{s_1+s_2+u}U^{-\beta+s_2+u}a^{-2+2u}  (2^{1-2u}-1)   \sideset{}{^*}\sum_{ k_1 \operatorname{odd}}\frac{1}{k_1^u}\nonumber\\
    &\times 
    L(1+s_1, f\otimes \chi_m) L(1+s_2, f \otimes \chi_m) Y(1/2+s_1,1/2+s_2,u;k_1,a) \mathcal{K}(s_1,s_2,u;k_1,\alpha,\beta)\,ds_1\,ds_2\,du.
\end{align*}
Introduce smoothed dyadic sums as in the proof of Lemma \ref{equ:reduce}, and move the lines of the integrals to $\Re(s_i)= -1/2+1/\log X$, $i=1,2$, and $\Re(u)=1+1/\log X$. Then
\begin{align}
     S(&k\neq 0, k_1 \operatorname{odd})\nonumber\\
     =&\sideset{}{^\#}\sum_{N_1} \sideset{}{^\#}\sum_{N_2} X\sum_{\substack{(a,2)=1\\ a<Y}}\mu(a)
     \frac{1}{(2\pi i)^5} \int_{(1+\frac{1}{\log X})}  \int_{(-\frac{1}{2}+\frac{1}{\log X})}\int_{(-\frac{1}{2}+\frac{1}{\log X})}
     \int_{(0)}\int_{(0)} (2^{1-2u}-1) \nonumber \\
     &\times 
     X^{s_1+s_2+u}U^{-\beta+s_2+u}a^{-2+2u}   N_1^{z_1-s_1-1/2} N_2^{z_2-s_2-1/2}\nonumber\\
    &\times 
     \sideset{}{^*}\sum_{ k_1 \operatorname{odd}}\frac{1}{k_1^u}
     \sum_{n_1=1}^\infty \frac{\lambda_f(n_1)\chi_m(n_1)}{n_1^{1/2+z_1}} V\left( \frac{n_1}{N_1} \right)\sum_{n_2=1}^\infty \frac{\lambda_f(n_2)\chi_m(n_2)}{n_2^{1/2+z_2}} V\left( \frac{n_2}{N_2} \right) \widetilde{G}(z_1-s_1-1/2)
   \nonumber \\
    &\times  \widetilde{G}(z_2-s_2-1/2) Y(1/2+s_1,1/2+s_2,u;k_1,a) \mathcal{K}(s_1,s_2,u;k_1,\alpha,\beta) \,dz_1\,dz_2 \,ds_1\,ds_2\,du.
    \label{equ:off-2}
\end{align} 
By Lemma \ref{lem-large-sieve}, 
\begin{align*}
   & \sideset{}{^*}\sum_{ k_1 \operatorname{odd}}\frac{1}{k_1^{1+1/\log X}}
     \left|\sum_{n_1=1}^\infty \frac{\lambda_f(n_1)\chi_m(n_1)}{n_1^{1/2+z_1}} V\left( \frac{n_1}{N_1} \right)\right|^2 \\
     &\ll \sideset{}{^\#}\sum_M \frac{1}{M^{1+1/\log X}} \sideset{}{^*}\sum_{ M\leq k_1 \leq 2M}
     \left|\sum_{n_1=1}^\infty \frac{\lambda_f(n_1)\chi_m(n_1)}{n_1^{1/2+z_1}} V\left( \frac{n_1}{N_1} \right)\right|^2 \\
     &\ll (\log X)(1+|\Im(z_1)|)^3\log (2+|\Im(z_1)|) .
\end{align*}
Substituting it in \eqref{equ:off-2} gives 
\begin{align*}
     S(k\neq 0, k_1 \operatorname{odd})
    & \ll X(\log X)^5 U^{1/2} \sideset{}{^\#}\sum_{N_1} \sideset{}{^\#}\sum_{N_2} N_1^{-1/\log X} N_2^{-1/\log X}\sum_{a<Y} \tau(a)
         \\
      &  \ll X(\log X)^{8} Y U^{1/2},
\end{align*}
which implies
\begin{lem}\label{lem-off}
We have
\begin{align*} 
     S(k\neq 0)
      &  \ll X(\log X)^{8} Y U^{1/2}.
\end{align*}
\end{lem}

\bigskip

\section{Complete the proof} \label{sec6}

Combining 
\eqref{equ:S}, \eqref{def:S(a<Y)+S(a>Y)}, 
\eqref{def:S(k=0)+S(kne0)}, Lemma \ref{lem:asyk=0} and Lemma \ref{lem-off},
\begin{align*}
    &\sideset{}{^*}\sum_{(d,2)=1}\left(\frac{8d}{2\pi}\right)^{-1}I_{\alpha}I_{\beta,U}\PhidX\\
    =& M(\alpha,\beta) + S_2(\alpha,\beta;Y)+ S_3(\alpha,\beta;Y)+O(X(\log \log X)^3) \\
        &+ O(X {|\alpha+\beta|}^{-1} )+ O(X(\log X)^{8} Y U^{1/2}).
\end{align*}
By Lemma  \ref{lemagY} and  \eqref{equ:upperS3},  $S_2(\alpha,\beta;Y), S_3(\alpha,\beta;Y)$ are holomorphic for $|\Re(\alpha)|,|\Re(\beta)|\leq 1/L $, and in the same region, they are bounded by 
\[
S_2(\alpha,\beta;Y),  S_3(\alpha,\beta;Y) \ll  X^{1+3|\Re(\alpha)|+3|\Re(\beta)|}Y^{-1} (\log X)^{50}.
\]

Recall $I_{\alpha}I_{\beta,U}$ is one of the terms in \eqref{equ:split}, and the asymptotic formulas of the mean values of $I_{\alpha,U}I_\beta$ and $ I_{\alpha,U}I_{\beta,U}$ in \eqref{equ:split}  can be written down directly by comparison. The contribution from the term  $R_{\alpha}R_{\beta}$ is given in Lemma \ref{equ:reduce}. We have therefore obtained an asymptotic formula for the mean value of  $I_{\alpha}I_{\beta}$ in \eqref{equ:split}. Moreover, we see $I_{\alpha}I_{\beta}$ is one of the terms in \eqref{eq_1}, and the asymptotic formulas for the other terms in \eqref{eq_1} can be obtained by symmetry. Now we use the notation $S(\alpha,\beta;Y)$ to denote the sum of all the contributions from $a>Y$  that appear in the above process analogous to $S_2(\alpha,\beta;Y),  S_3(\alpha,\beta;Y)$ as shown in \eqref{def:S(a<Y)+S(a>Y)} and \eqref{def-s3}. By taking $U= Y^{-2}(\log X)^{-20}$, 
\begin{align}
&\sideset{}{^*}\sum_{(d,2)=1}\left(\frac{8d}{2\pi}\right)^{-1}\Lambda(1/2+\alpha,\fod)\Lambda(1/2+\beta,\fod)\PhidX\nonumber\\
=& M(\alpha,\beta) - M(\alpha,-\beta)-M(-\alpha,\beta)+M(-\alpha,-\beta)   
+ S(\alpha,\beta;Y) \nonumber\\
&+ O(X(\log \log X)^3) + O(X {|\alpha+\beta|}^{-1} ),
\label{equ:final}
\end{align}
where in the region $|\Re(\alpha)|,|\Re(\beta)|\leq 1/L $, $S(\alpha,\beta;Y)$ is holomorphic and 
\begin{align}
S(\alpha,\beta;Y)\ll  X^{1+3|\Re(\alpha)|+3|\Re(\beta)|}Y^{-1} (\log X)^{50}.
\label{equ:Sbd}
\end{align}

Now we prove Theorem \ref{main-thm}. By \eqref{eq:AFE of L'}, 
\begin{align*}
       &\sideset{}{^*}\sum_{(d,2)=1}   L'(1/2,\fod)^2    \PhidX\nonumber \\
       =&\Gamma\left(\frac{\kappa}{2}\right)^{-2}\frac{4}{(2\pi i)^2}\int_{(\frac{1}{4L})}\int_{(\frac{1}{2L})}
        \sideset{}{^*}\sum_{(d,2)=1}  \left(\frac{8d}{2\pi}\right)^{-1}
       \Lambda(1/2+s_1,\fod) \Lambda(1/2+s_2,\fod)
       \PhidX\nonumber\\
       &\times 
       e^{s_1^2+s_2^2}\frac{ds_1}{s_1^2}  \frac{ds_2}{s_2^2}.
    \end{align*}
Inserting \eqref{equ:final} into the above implies 
\begin{align}
       &\sideset{}{^*}\sum_{(d,2)=1}   L'(1/2,\fod)^2    \PhidX\nonumber \\
       =&\Gamma\left(\frac{\kappa}{2}\right)^{-2}\frac{4}{(2\pi i)^2}\int_{(\frac{1}{4L})}\int_{(\frac{1}{2L})}
        [M(s_1,s_2) - M(s_1,-s_2)-M(-s_1,s_2)+M(-s_1,-s_2)\nonumber\\
        &+S(s_1,s_2;Y)] e^{s_1^2+s_2^2}\frac{ds_1}{s_1^2}  \frac{ds_2}{s_2^2}+ O(X (\log \log X)^5),
        \label{equ:seclast}
\end{align}
where  we have used the bound
\[
\int_{(\frac{1}{L})} \frac{1}{(1+|s|^{10})|s|^2}  \, |ds|  \ll \log \log X.
\]
We first evaluate the contribution from $S(s_1,s_2;Y)$. Move the line of the integral over $s_2 $ to $\Re(s_2)=1/(4\log X)$, and then to $\Re(s_1)=1/(2\log X)$ for the integral over $s_1$.  This process encounters no  poles since $S(s_1,s_2;Y)$ is holomorphic in the region $|\Re(\alpha)|,|\Re(\beta)|\leq 1/L $. By \eqref{equ:Sbd},
\begin{align*}
       &\frac{1}{(2\pi i)^2}\int_{(\frac{1}{4\log X})}\int_{(\frac{1}{2\log X})}
       S(s_1,s_2;Y)  e^{s_1^2+s_2^2}\frac{ds_1}{s_1^2}  \frac{ds_2}{s_2^2}
       \ll  X Y^{-1} (\log X)^{52} \ll X(\log X)^{-1},
\end{align*}
where we have taken $Y=(\log X)^{53} $. In addition, we see 
\begin{align*}
    \frac{1}{(2\pi i)^2}\int_{(\frac{1}{4L})}\int_{(\frac{1}{2L})}
        (M(s_1,s_2) - M(s_1,-s_2))  e^{s_1^2+s_2^2}
        \frac{ds_1}{s_1^2}  \frac{ds_2}{s_2^2}
   & =
   \frac{1}{2\pi i}\int_{(\frac{1}{2L})}
        M^{(0,1)}(s_1,0)  e^{s_1^2}
        \frac{ds_1}{s_1^2},\\
    \frac{1}{(2\pi i)^2}\int_{(\frac{1}{4L})}\int_{(\frac{1}{2L})}
        (-M(-s_1,s_2) + M(-s_1,-s_2))  e^{s_1^2+s_2^2}
        \frac{ds_1}{s_1^2}  \frac{ds_2}{s_2^2}
   & =
   -\frac{1}{2\pi i}\int_{(\frac{1}{2L})}
        M^{(0,1)}(-s_1,0)  e^{s_1^2}
        \frac{ds_1}{s_1^2} .
\end{align*}
Also,  
\begin{align*}
  \frac{1}{2\pi i}\int_{(\frac{1}{2L})}
        M^{(0,1)}(s_1,0)
       e^{s_1^2} \frac{ds_1}{s_1^2}
    -\frac{1}{2\pi i}\int_{(\frac{1}{2L})}
        M^{(0,1)}(-s_1,0) e^{s_1^2}
        \frac{ds_1}{s_1^2} 
= \underset{s_1=0} {\operatorname{Res}}\left( M^{(0,1)}(s_1,0) e^{s_1^2}\frac{1}{s_1^2}\right).
\end{align*}
Here $M^{(k,\ell)}(s_1,s_2) := \frac{\partial^{k+\ell} }{\partial s_1^k \partial s_2^\ell} M(s_1,s_2)$. Recall the definition of $M(\alpha,\beta)$ in \eqref{equ:dia-5} and write 
\[
M(\alpha,\beta) =:X^{1+\alpha+\beta} \zeta(1+\alpha+\beta) T(\alpha,\beta).
\]
Then 
\begin{align*}
M^{(0,1)}(s_1,0) =& X^{1+s_1}(\log X) \zeta(1+s_1) T(s_1,0) + X^{1+s_1} \zeta'(1+s_1) T(s_1,0)\\
&+ X^{1+s_1} \zeta(1+s_1) T^{(0,1)}(s_1,0).
\end{align*}
We write 
\begin{align*}
X^{s_1}& = 1+ (\log X) s_1 + \frac{1}{2!}(\log X)^2 s_1^2 + \frac{1}{3!}(\log X)^3 s_1^3+ \cdots,  \\
\zeta(1+s_1) &=  \frac{1}{s_1} + \gamma_0 + \gamma_1 s_1 + \cdots, \\
T(s_1,0) & = a_0+ a_1 s_1 + a_2 s_1^2\cdots,\\
e^{s_1^2} & = 1+  s_1^2 + \cdots .
\end{align*}
Therefore,
\begin{align*}
    \underset{s_1=0} {\operatorname{Res}} \left( M^{(0,1)}(s_1,0) e^{s_1^2}\frac{1}{s_1^2}\right) = C_3 X (\log X)^3 + C_2 X (\log X)^2 +C_1 X \log X + C_0 X,
\end{align*}
where  $C_i$ can be computed precisely, in particular,
\begin{align*}
    C_3& = \frac{1}{6\pi^2}
     L(1,\operatorname{sym}^2 f ) ^3
     Z_1(0,0)\Gamma(\tfrac{\kappa}{2})^2
    \widetilde{\Phi}(1). 
\end{align*}
By \eqref{equ:seclast} and the discussion following it, 
\begin{align*}
       \sideset{}{^*}\sum_{(d,2)=1}   L'(1/2,\fod)^2   \PhidX
       = c_3 X (\log X)^3 + c_2 X (\log X)^2 +c_1 X \log X +  O(X (\log \log X)^5),
\end{align*}
where $c_i = 4\Gamma(\frac{\kappa}{2})^{-2}C_i$, $i=1,2,3$. This completes the proof of  Theorem \ref{main-thm}.


\bibliographystyle{plain}
 \bibliography{ams-reference}

@article {Zhou,
    AUTHOR = {Zhou, Z.},
     TITLE = {Moment of derivatives of quadratic twists of modular {$L$}-Functions},
   JOURNAL = {{\rm To appear in}  Algebra  Number Theory},
  FJOURNAL = {},
      YEAR = {arXiv:2503.14680},
     PAGES = {},
      ISSN = {},
   MRCLASS = {},
  MRNUMBER = {},
MRREVIEWER = {},
       DOI = {},
       URL = {https://doi.org/10.19086/da},
}

@article {NSW02,
    AUTHOR = {Ng, N. and Shen, Q. and Wong, P.-J. },
     TITLE = {{The eighth moment of the Riemann zeta function}},
   JOURNAL = {J. Eur. Math. Soc. (JEMS)},
  FJOURNAL = {},
      YEAR = {pubslished online, 2025},
     PAGES = {},
      ISSN = {},
   MRCLASS = {},
  MRNUMBER = {},
MRREVIEWER = {},
       DOI = {},
       URL = {https://doi.org/10.19086/da},
}

@book {I-K,
    AUTHOR = {Iwaniec, H. and Kowalski, E.},
     TITLE = {Analytic number theory},
    SERIES = {American Mathematical Society Colloquium Publications},
    VOLUME = {53},
 PUBLISHER = {American Mathematical Society, Providence, RI},
      YEAR = {2004},
     PAGES = {xii+615},
      ISBN = {0-8218-3633-1},
   MRCLASS = {11-02 (11Fxx 11Lxx 11Mxx 11Nxx)},
  MRNUMBER = {2061214},
MRREVIEWER = {K.\ Soundararajan},
       DOI = {10.1090/coll/053},
       URL = {https://doi.org/10.1090/coll/053},
}

@article {Li,
    AUTHOR = {Li, X.},
     TITLE = {Moments of quadratic twists of modular {$L$}-functions},
   JOURNAL = {Invent. Math.},
  FJOURNAL = {Inventiones Mathematicae},
    VOLUME = {237},
      YEAR = {2024},
    NUMBER = {2},
     PAGES = {697--733},
      ISSN = {0020-9910,1432-1297},
   MRCLASS = {11F67 (11F11)},
  MRNUMBER = {4768632},
       DOI = {10.1007/s00222-024-01265-1},
       URL = {https://doi.org/10.1007/s00222-024-01265-1},
}

@article {Kumar,
    AUTHOR = {Kumar, S. and Mallesham, K. and Sharma, P. and
              Singh, S. K.},
     TITLE = {Moments of derivatives of modular {$L$}-functions},
   JOURNAL = {Q. J. Math.},
  FJOURNAL = {The Quarterly Journal of Mathematics},
    VOLUME = {75},
      YEAR = {2024},
    NUMBER = {2},
     PAGES = {715--734},
      ISSN = {0033-5606,1464-3847},
   MRCLASS = {11F66 (11N37)},
  MRNUMBER = {4765788},
       DOI = {10.1093/qmath/haae028},
       URL = {https://doi.org/10.1093/qmath/haae028},
}

@article {Bump-F-H,
    AUTHOR = {Bump, D. and Friedberg, S. and Hoffstein, J.},
     TITLE = {Nonvanishing theorems for {$L$}-functions of modular forms and
              their derivatives},
   JOURNAL = {Invent. Math.},
  FJOURNAL = {Inventiones Mathematicae},
    VOLUME = {102},
      YEAR = {1990},
    NUMBER = {3},
     PAGES = {543--618},
      ISSN = {0020-9910},
   MRCLASS = {11F66 (11F67 11G05 11G40)},
  MRNUMBER = {1074487},
MRREVIEWER = {Alexey A. Panchishkin},
       DOI = {10.1007/BF01233440},
       URL = {https://doi.org/10.1007/BF01233440},
}

@article {Chandee,
    AUTHOR = {Chandee, V.},
     TITLE = {On the correlation of shifted values of the {R}iemann zeta
              function},
   JOURNAL = {Q. J. Math.},
  FJOURNAL = {The Quarterly Journal of Mathematics},
    VOLUME = {62},
      YEAR = {2011},
    NUMBER = {3},
     PAGES = {545--572},
      ISSN = {0033-5606},
   MRCLASS = {11M06},
  MRNUMBER = {2825471},
MRREVIEWER = {Matthew P. Young},
       DOI = {10.1093/qmath/haq008},
       URL = {https://doi.org/10.1093/qmath/haq008},
}

@article {Conrey-Farmer-Keating-Rubinstein-Snaith,
    AUTHOR = {Conrey, J. B. and Farmer, D. W. and Keating, J. P. and
              Rubinstein, M. O. and Snaith, N. C.},
     TITLE = {Integral moments of {$L$}-functions},
   JOURNAL = {Proc. London Math. Soc. (3)},
  FJOURNAL = {Proceedings of the London Mathematical Society. Third Series},
    VOLUME = {91},
      YEAR = {2005},
    NUMBER = {1},
     PAGES = {33--104},
      ISSN = {0024-6115},
   MRCLASS = {11M26},
  MRNUMBER = {2149530},
MRREVIEWER = {K. Soundararajan},
       DOI = {10.1112/S0024611504015175},
       URL = {https://doi.org/10.1112/S0024611504015175},
}

@article {Kolyvagin,
    AUTHOR = {Kolyvagin, V. A.},
     TITLE = {Finiteness of {$E({\bf Q})$} and {III}{{$(E,{\bf Q})$}} for a
              subclass of {W}eil curves},
   JOURNAL = {Izv. Akad. Nauk SSSR Ser. Mat.},
  FJOURNAL = {Izvestiya Akademii Nauk SSSR. Seriya Matematicheskaya},
    VOLUME = {52},
      YEAR = {1988},
    NUMBER = {3},
     PAGES = {522--540, 670--671},
      ISSN = {0373-2436},
   MRCLASS = {11G05 (11G40 14G25 14K07)},
  MRNUMBER = {954295},
MRREVIEWER = {Reinhard B\"{o}lling},
       DOI = {10.1070/IM1989v032n03ABEH000779},
       URL = {https://doi.org/10.1070/IM1989v032n03ABEH000779},
}

@article {Murty-Murty,
    AUTHOR = {Murty, M. R. and Murty, V. K.},
     TITLE = {Mean values of derivatives of modular {$L$}-series},
   JOURNAL = {Ann. of Math. (2)},
  FJOURNAL = {Annals of Mathematics. Second Series},
    VOLUME = {133},
      YEAR = {1991},
    NUMBER = {3},
     PAGES = {447--475},
      ISSN = {0003-486X},
   MRCLASS = {11F67 (11G05 11G40)},
  MRNUMBER = {1109350},
MRREVIEWER = {Daniel Bump},
       DOI = {10.2307/2944316},
       URL = {https://doi.org/10.2307/2944316},
}

@article {Petrow,
    AUTHOR = {Petrow, I.},
     TITLE = {Moments of {$L'(\frac12)$} in the family of quadratic twists},
   JOURNAL = {Int. Math. Res. Not. IMRN},
  FJOURNAL = {International Mathematics Research Notices. IMRN},
      YEAR = {2014},
    NUMBER = {6},
     PAGES = {1576--1612},
      ISSN = {1073-7928},
   MRCLASS = {11N37 (11A25)},
  MRNUMBER = {3180602},
MRREVIEWER = {Olivier Bordell\`es},
       DOI = {10.1093/imrn/rns265},
       URL = {https://doi.org/10.1093/imrn/rns265},
}

@article {Sound-Young,
    AUTHOR = {Soundararajan, K. and Young, M. P.},
     TITLE = {The second moment of quadratic twists of modular
              {$L$}-functions},
   JOURNAL = {J. Eur. Math. Soc. (JEMS)},
  FJOURNAL = {Journal of the European Mathematical Society (JEMS)},
    VOLUME = {12},
      YEAR = {2010},
    NUMBER = {5},
     PAGES = {1097--1116},
      ISSN = {1435-9855},
   MRCLASS = {11F66 (11F67)},
  MRNUMBER = {2677611},
MRREVIEWER = {D. R. Heath-Brown},
       DOI = {10.4171/JEMS/224},
       URL = {https://doi.org/10.4171/JEMS/224},
}

@article {Young,
    AUTHOR = {Young, M. P.},
     TITLE = {The third moment of quadratic {D}irichlet {$L$}-functions},
   JOURNAL = {Selecta Math. (N.S.)},
  FJOURNAL = {Selecta Mathematica. New Series},
    VOLUME = {19},
      YEAR = {2013},
    NUMBER = {2},
     PAGES = {509--543},
      ISSN = {1022-1824},
   MRCLASS = {11M06 (11A25 11N37)},
  MRNUMBER = {3090236},
MRREVIEWER = {Olivier Bordell\`es},
       DOI = {10.1007/s00029-012-0104-4},
       URL = {https://doi.org/10.1007/s00029-012-0104-4},
}



%


\end{document}